\pdfoutput=1
\RequirePackage{ifpdf}
\ifpdf % We are running pdfTeX in pdf mode
\documentclass[pdftex]{sigma}
\else
\documentclass{sigma}
\fi

\numberwithin{equation}{section}

\newtheorem{Theorem}{Theorem}[section]
\newtheorem{Corollary}[Theorem]{Corollary}
\newtheorem{Lemma}[Theorem]{Lemma}
\newtheorem{Proposition}[Theorem]{Proposition}
\newtheorem{Conjecture}[Theorem]{Conjecture}
 { \theoremstyle{definition}
\newtheorem{Example}[Theorem]{Example}
\newtheorem{Remark}[Theorem]{Remark}
\newtheorem{Openproblem}[Theorem]{Open Problem}
 }

\usepackage{tikz}\usetikzlibrary{trees,snakes,shapes,arrows,matrix,calc}

\newcommand{\N}{\mathbb{N}} % naturals
\newcommand{\F}{\mathcal{F}} % field

\newcommand{\Fibonomial}[2]{{{#1 + #2} \choose #2}_\F} % Fibonomial number
\newcommand{\qFibonomial}[2]{{\left[\!\begin{smallmatrix} #1 + #2 \\ #2 \end{smallmatrix}\!\right]}_\F} % (m+n choose n)
\newcommand{\qFibonomialBennet}[2]{{\left[\!\begin{smallmatrix} #1 \\ #2 \end{smallmatrix}\!\right]}_\F} % (n choose k)
\newcommand{\BigqFibonomial}[2]{\genfrac{[}{]}{0pt}{0}{#1 + #2}{#2}_\F} % Big q Fibonomial number
\newcommand{\BigqFibonomialBennet}[2]{\genfrac{[}{]}{0pt}{0}{#1}{#2}_\F} % Big q Fibonomial number n choose k
\newcommand{\Fibofactorial}[1]{F_{#1}^!} %
\newcommand{\qFibofactorial}[1]{[F_{#1}]_q^!} %
\newcommand{\FiboCatalan}[1]{\frac{1}{F_{#1+1}}{2#1 \choose #1}_\F} %
\newcommand{\mnCat}[2]{\operatorname{Cat}_{#1,#2}} %
\newcommand{\mnFCat}[2]{\operatorname{FCat}_{#1,#2}} %
\newcommand{\nFCat}[1]{\operatorname{FCat}_{#1}} %
\newcommand{\mnCatalan}[2]{\frac{1}{#1+#2}{#1+#2 \choose #2}} %
\newcommand{\mnFiboCatalan}[2]{\frac{1}{F_{m+n}}\Fibonomial{m}{n}} %
\newcommand{\mnqFiboCatalan}[2]{\frac{1}{[F_{m+n}]}\qFibonomial{m}{n}} %
\newcommand{\BigmnqFiboCatalan}[2]{\frac{1}{[F_{m+n}]}\BigqFibonomial{m}{n}} %

\newcommand{\elliptic}[1]{[#1]_{a,b;q,p}}	% elliptic number
\newcommand{\ellipticFibonomial}[2]{{\left[\!\begin{smallmatrix} #1 + #2 \\ #2 \end{smallmatrix}\!\right]}_{\F_{a,b;q,p}}} % (m+n choose n)
\newcommand{\BigellipticFibonomial}[2]{\genfrac{[}{]}{0pt}{0}{#1 + #2}{#2}_{\F_{a,b;q,p}}} %
\newcommand{\BigellipticFibonomialBennet}[2]{\genfrac{[}{]}{0pt}{0}{#1}{#2}_{\F_{a,b;q,p}}} %
\newcommand{\ellipticFiboFactorial}[1]{[F_{#1}]_{a,b;q,p}^!} %
\newcommand{\eweight}{{\widetilde \omega}}
\newcommand{\qBinomial}[3]{\genfrac{[}{]}{0pt}{0}{#1}{#2}_{#3}} % Big q Binomial

\newcommand{\specialdominocolor}{orange!20} % special vertical domino color

\newcommand{\pathwidth}{0.7mm}
\newcommand{\boxsize}{1.8em}

\newcommand{\grid}[2]{%#1=width,#2=height
\draw (0,0) grid (#1,#2);
}

\newcommand{\horizontalDomino}[2]{
\draw (#1,#2) rectangle +(-2,-1);
\draw (#1,#2)++(-1,0) -- +(0,-1);
\draw (#1,#2)++(-0.5,-0.5) -- +(-1,0);
\fill (#1,#2)+(-0.5,-0.5) circle (2.5pt);
\fill (#1,#2)+(-1.5,-0.5) circle (2.5pt);
}

\newcommand{\coloredhorizontalDomino}[3]{
	\draw (#1,#2) rectangle +(-2,-1);
	\draw (#1,#2)++(-1,0) -- +(0,-1);
	\draw [color=#3] (#1,#2)++(-0.5,-0.5) -- +(-1,0);
	\fill [color=#3] (#1,#2)+(-0.5,-0.5) circle (2.5pt);
	\fill [color=#3] (#1,#2)+(-1.5,-0.5) circle (2.5pt);
}

\newcommand{\verticalDomino}[2]{
\draw (#1,#2) rectangle +(-1,-2);
\draw (#1,#2)++(0,-1) -- +(-1,0);
\draw (#1,#2)++(-0.5,-0.5) -- +(0,-1);
\fill (#1,#2)+(-0.5,-0.5) circle (2.5pt);
\fill (#1,#2)+(-0.5,-1.5) circle (2.5pt);
}

\newcommand{\specialDomino}[2]{
\draw[fill=\specialdominocolor] (#1,#2) rectangle +(-1,-2);
\draw (#1,#2)++(0,-1) -- +(-1,0);
\draw (#1,#2)++(-0.5,-0.5) -- +(0,-1);
\fill (#1,#2)+(-0.5,-0.5) circle (2.5pt);
\fill (#1,#2)+(-0.5,-1.5) circle (2.5pt);
}

\newcommand{\specialhorizontalDomino}[2]{
	\draw[fill=\specialdominocolor]  (#1,#2) rectangle +(-2,-1);
	\draw (#1,#2)++(-1,0) -- +(0,-1);
	\draw (#1,#2)++(-0.5,-0.5) -- +(-1,0);
	\fill (#1,#2)+(-0.5,-0.5) circle (2.5pt);
	\fill (#1,#2)+(-1.5,-0.5) circle (2.5pt);
}

\newcommand{\monomino}[2]{
\draw (#1,#2) rectangle +(-1,-1);
\fill (#1,#2)+(-0.5,-0.5) circle (2.5pt);
}

\newcommand{\labeledGrid}[3]{%#1=width,#2=height,#3=Fibonacci sequence
\draw (0,0) grid (#1,#2);
\foreach \x in {1,...,#1}
	\node at (\x-0.5,-0.5) {$\pgfmathparse{\csname#3\endcsname[\x]}\pgfmathresult$};
\foreach \y in {1,...,#2}
	\node at (-0.5,\y-0.5) {$\pgfmathparse{\csname#3\endcsname[\y]}\pgfmathresult$};
}

\newcommand{\ordinaryLabeledGrid}[2]{%#1=width,#2=height
	\draw (0,0) grid (#1,#2);
	\foreach \x in {1,...,#1}
	\node at (\x-0.5,-0.5) {$\x$};
	\foreach \y in {1,...,#2}
	\node at (-0.5,\y-0.5) {$\y$};
}

\newcommand{\labeledHorizontalDomino}[3]{
	\draw (#1,#2) rectangle +(-2,-1);
	\draw (#1,#2)++(-0.5,-0.5) -- +(-1,0);
	\fill (#1,#2)+(-0.5,-0.5) circle (2.5pt) node[above=-0.03cm, inner sep=1.5pt] {\footnotesize $q^{\pgfmathparse{\csname#3\endcsname[#1]}\pgfmathresult  \cdot  \pgfmathparse{\csname#3\endcsname[#2]}\pgfmathresult}$};
	\fill (#1,#2)+(-1.5,-0.5) circle (2.5pt);
}

\newcommand{\labeledVerticalDomino}[3]{
	\draw (#1,#2) rectangle +(-1,-2);
	\draw (#1,#2)++(-0.5,-0.5) -- +(0,-1);
	\fill (#1,#2)+(-0.5,-0.5) circle (2.5pt) node[above=-0.03cm, inner sep=1.5pt] {\footnotesize $q^{\pgfmathparse{\csname#3\endcsname[#1]}\pgfmathresult  \cdot  \pgfmathparse{\csname#3\endcsname[#2]}\pgfmathresult}$};
	\fill (#1,#2)+(-0.5,-1.5) circle (2.5pt);
}

\newcommand{\labeledSpecialDomino}[3]{
\draw[fill=\specialdominocolor] (#1,#2) rectangle +(-1,-2);
\draw (#1,#2)++(0,-1) -- +(-1,0);
\draw (#1,#2)++(-0.5,-0.5) -- +(0,-1);
\fill (#1,#2)+(-0.5,-0.5) circle (2.5pt);
\fill (#1,#2)+(-0.5,-1.5) circle (2.5pt);
}

\newcommand{\labeledSpecialDominoLabeling}[3]{
	\node[above=-0.03cm, inner sep=1.5pt] at (#1-0.5,#2-0.5) {\footnotesize $q^{\pgfmathparse{\csname#3\endcsname[#1+1]}\pgfmathresult  \cdot  \pgfmathparse{\csname#3\endcsname[#2]}\pgfmathresult}$};
}

\newcommand{\ellipticLabeledHorizontalDomino}[2]{
	\draw (#1,#2) rectangle +(-2,-1);
	\draw (#1,#2)++(-0.5,-0.5) -- +(-1,0);
	\fill (#1,#2)+(-0.5,-0.5) circle (2.5pt);
	\node[above=0.1cm, inner sep=1.5pt] at (#1-0.5,#2-0.5) {\scriptsize $\omega_1 (#1 , #2)$};
	\fill (#1,#2)+(-1.5,-0.5) circle (2.5pt);
}

\newcommand{\ellipticLabeledVerticalDomino}[2]{
	\draw (#1,#2) rectangle +(-1,-2);
	\draw (#1,#2)++(-0.5,-0.5) -- +(0,-1);
	\fill (#1,#2)+(-0.5,-0.5) circle (2.5pt) node[above=0.1cm, inner sep=1.5pt] {\scriptsize $\omega_1 (#2 , #1)$};
	\fill (#1,#2)+(-0.5,-1.5) circle (2.5pt);
}

\newcommand{\ellipticLabeledSpecialDomino}[2]{
	\draw[fill=\specialdominocolor] (#1,#2) rectangle +(-1,-2);
	\draw (#1,#2)++(0,-1) -- +(-1,0);
	\draw (#1,#2)++(-0.5,-0.5) -- +(0,-1);
	\fill (#1,#2)+(-0.5,-0.5) circle (2.5pt);
	\fill (#1,#2)+(-0.5,-1.5) circle (2.5pt);
}
\newcommand{\ellipticLabeledSpecialDominoLabeling}[2]{
	\node[above=0.1cm, inner sep=1.5pt] at (#1-0.5,#2-0.5) {\scriptsize $\omega_2 (#1 , #2)$};
}

\newcommand{\fiboSquare}[3]{%#1=x-position of lower-left corner,#2=y-position of ll-corner,#3=width
	\draw[thick] (#1,#2) rectangle +(#3,#3);
}

\newcommand{\tilings}[2]{\mathcal{T}_{#1,#2}}
\newcommand{\staircaseTilings}[2]{\mathcal{S}_{#1,#2}}

\begin{document}
\allowdisplaybreaks

\newcommand{\arXivNumber}{1911.12785}

\renewcommand{\thefootnote}{}

\renewcommand{\PaperNumber}{076}

\FirstPageHeading

\ShortArticleName{Elliptic and $q$-Analogs of the Fibonomial Numbers}

\ArticleName{Elliptic and $\boldsymbol{q}$-Analogs of the Fibonomial Numbers\footnote{This paper is a~contribution to the Special Issue on Elliptic Integrable Systems, Special Functions and Quantum Field Theory. The full collection is available at \href{https://www.emis.de/journals/SIGMA/elliptic-integrable-systems.html}{https://www.emis.de/journals/SIGMA/elliptic-integrable-systems.html}}}

\Author{Nantel BERGERON~$^{\dag}$, Cesar CEBALLOS~$^{\ddag}$ and Josef K\"USTNER~$^{\S}$}

\AuthorNameForHeading{N.~Bergeron, C.~Ceballos and J.~K\"ustner}

\Address{$^{\dag}$~Department of Mathematics and Statistics, York University, Toronto, Canada}
\EmailD{\href{mailto:bergeron@mathstat.yorku.ca}{bergeron@mathstat.yorku.ca}}
\URLaddressD{\url{http://www.math.yorku.ca/bergeron/}}

\Address{$^{\ddag}$~Institute of Geometry, TU Graz, Graz, Austria}
\EmailD{\href{mailto:cesar.ceballos@tugraz.at}{cesar.ceballos@tugraz.at}}
\URLaddressD{\url{http://www.geometrie.tugraz.at/ceballos/}}

\Address{$^{\S}$~Faculty of Mathematics, University of Vienna, Vienna, Austria}
\EmailD{\href{mailto:josef.kuestner@univie.ac.at}{josef.kuestner@univie.ac.at}}
\URLaddressD{\url{http://homepage.univie.ac.at/josef.kuestner/}}

\ArticleDates{Received March 14, 2020, in final form July 29, 2020; Published online August 13, 2020}

\Abstract{In 2009, Sagan and Savage introduced a combinatorial model for the Fibonomial numbers, integer numbers that are obtained from the binomial coefficients by replacing each term by its corresponding Fibonacci number. In this paper, we present a combinatorial description for the $q$-analog and elliptic analog of the Fibonomial numbers. This is achieved by introducing some $q$-weights and elliptic weights to a slight modification of the combinatorial model of Sagan and Savage.}

\Keywords{Fibonomial; Fibonacci; $q$-analog; elliptic analog; weighted enumeration}

\Classification{11B39; 05A30; 05A10}

\renewcommand{\thefootnote}{\arabic{footnote}}
\setcounter{footnote}{0}

\section{Introduction}

The Fibonacci sequence $0, 1, 1, 2, 3, 5, 8, 13, 21, 34, \dots$ is
one of the most important and beautiful sequences in mathematics.
It starts with the numbers $F_0=0$ and $F_1=1$, and is recursively defined by the formula $F_n=F_{n-1}+F_{n-2}$.

Fibonacci analogs of famous numbers, such as the binomial coefficients and Catalan numbers
\[
{m+n \choose n}=\frac{(m+n)!}{m!\cdot n!} \qquad {\text{and}} \qquad \frac{1}{n+1}{2n \choose n},
\]
have intrigued some mathematicians over the last few years~\cite{generalized_amdeberhan_2014,combinatorial_benjamin_2014,combinatorial_bennett_2018,thefractal_chen_2014,combinatorial_sagan_2010,combinatorial_tirrell_2019}.
The \emph{Fibonomial} and \emph{Fibo-Catalan} numbers are defined, respectively, as
\[
\Fibonomial{m}{n} := \frac{\Fibofactorial{m+n}}{\Fibofactorial{m}\cdot \Fibofactorial{n}} \qquad {\text{and}} \qquad \FiboCatalan{n},
\]
where $\Fibofactorial{n}:= \prod\limits_{k=1}^n F_k$
is the Fibonacci analog of $n!$.
These rational expressions turn out to be positive integers.
In~\cite{combinatorial_sagan_2010}, Sagan and Savage introduced a combinatorial model to interpret
the Fibonomial numbers in terms of certain tilings of an $m\times n$ rectangle.
A \emph{path-domino tiling} of an $m\times n$ rectangle is a tiling with monominos and dominos and a lattice path from $(0,0)$ to $(m,n)$ is specified, and such that:
\pagebreak
\begin{itemize}\itemsep=0pt
	\item all tiles above the path are either monominos or horizontal dominos;
	\item all tiles below the path are either monominos or vertical dominos; and
	\item all tiles that touch the path from below are vertical dominos.
\end{itemize}
We call these last tiles touching the path from below \emph{special vertical dominos},
and denote by~$\tilings{m}{n}$ the collection of all path-domino tilings of an $m\times n$ rectangle.
An example is illustrated on the left of Fig.~\ref{fig_pathdominotiling}.
The following result is a special case of~\cite[Theorem~3]{combinatorial_sagan_2010}.\footnote{The path-domino tilings here are a slight modification of the tilings used in~\cite{combinatorial_sagan_2010}.	The only difference is that in~\cite{combinatorial_sagan_2010} the tiles below the path touching the bottom of the rectangle are required to be vertical dominos, while here this condition is required for the tiles below the path touching the path itself. This modification is essential to make our combinatorial model work.}

\begin{Theorem}[{\cite{combinatorial_sagan_2010}}]
	The Fibonomial number $\Fibonomial{m}{n}$ counts the number of path-domino tilings of an $m\times n$ rectangle.
\end{Theorem}

The main objective of this paper is to present a $q$-analog and an elliptic analog generalization of this result.
The resulting $q$-Fibonomial and elliptic Fibonomial numbers count the number of path-domino tilings of an $m\times n$ rectangle
according to their $q$-weights and elliptic weights, respectively.

\section[$q$-analog of the Fibonomial numbers]{$\boldsymbol{q}$-analog of the Fibonomial numbers}\label{sec_q}
We denote by $\N:= \{1,2,3,\dots\}$ the set of natural numbers.
The \emph{$q$-analog of $n\in \N$} is defined as
\[
[n]_q := 1+q+q^2+\dots + q^{n-1}.
\]
The evaluation of this polynomial at $q=1$ recovers the number $n$.
To simplify notation, we sometimes omit the subindex $q$ when it is clear from the context.
Before studying the $q$-analog of the Fibonomial numbers, let us recall some useful and known straightforward lemmas.

\begin{Lemma} \label{lem_qidentities}
	For $m,n\in \N$, the following identities hold:
	\begin{gather}
	[m+n]_q = [m]_q + q^m[n]_q, \label{lem_qidentities1} \\
	[m\cdot n]_q = [m]_q [n]_{q^m}. \nonumber %\label{lem_qidentities2}
	\end{gather}
\end{Lemma}

It is well known that the Fibonacci number $F_n$ counts the number of tilings of an \emph{$(n-1)$-strip} (a rectangle with diagonal endpoints $(0,0)$ and $(n-1,1)$) using dominos and monominos. Given such a tiling $T$, we define the \emph{weight $\omega(T)$} of $T$ as the product of the weights of its tiles, where
a monomino has weight~1 and a domino whose top-right coordinate is $(i,1)$ has weight $q^{F_i}$.
The weight of a 0-strip is by definition equal to~1.

\begin{Lemma}[{cf.\ \cite{combinatorial_sagan_2010}}]\label{lem_q_analog_Fibonacci}
	For $n\in \N$, the $q$-analog of the Fibonacci numbers\footnote{The elliptic and $q$-analogs of the Fibonacci numbers we use are different from the analogs considered, e.g., in~\cite{elliptic_schlosser_2018}.} can be computed as
	\[
	[F_n]_q = \sum_T {\omega(T)},
	\]
	where $	[F_n]_q = 1 + q + q^2 + \dots + q^{F_n -1} $ and the sum ranges over all tilings of an $(n-1)$-strip using dominos and monominos.
\end{Lemma}
\begin{proof}The result is clearly true for $n=1,2$. Let $n>2$, applying equation~\eqref{lem_qidentities1} from Lemma~\ref{lem_qidentities} we get
	\begin{gather*}
	[F_n]_q = [F_{n-1}+F_{n-2}]_q = [F_{n-1}]_q + q^{F_{n-1}}[F_{n-2}]_q.
	\end{gather*}
	By induction, the first term of this sum corresponds to the tilings of an $(n-1)$-strip that finish with a monomino,
	while the second term to the tilings of an $(n-1)$-strip that finish with a~domino.
\end{proof}

\begin{Lemma}\label{lem_qFibIdentities}
	For $m,n\in \N$, the following identities hold:
	\begin{gather}
	F_{m+n} = F_nF_{m+1} + F_mF_{n-1}, \label{lem_qFibIdentities1} \\
	[F_{m+n}]_q = [F_n]_q[F_{m+1}]_{q^{F_n}} + q^{F_nF_{m+1}}[F_m]_q[F_{n-1}]_{q^{F_m}}. \label{lem_qFibIdentities2}
	\end{gather}
\end{Lemma}
\begin{proof}
	Equation~\eqref{lem_qFibIdentities1} is a well known identity for Fibonacci numbers; see for instance~\cite[Lemma~1]{combinatorial_sagan_2010}.	
	The Fibonacci number $F_{m+n}$ counts the number of tilings of an $(m+n-1)$-strip with monominos and dominos.
	These tilings can be subdivided into two types: those containing a domino which is cut in two by the line $x=m$, and all other tilings.
	The first kind is counted by~$F_mF_{n-1}$ (the number of tilings of an $(m-1)$-strip times the number of tilings of an $(n-2)$-strip),
	while
	the second kind is counted by~$F_nF_{m+1}$ (the number of tilings of an $m$-strip times the number of tilings of an $(n-1)$-strip).
	Therefore, equation~\eqref{lem_qFibIdentities1} follows.
	
	Applying Lemma~\ref{lem_qidentities} to~\eqref{lem_qFibIdentities1} leads to equation~\eqref{lem_qFibIdentities2}.
\end{proof}

For $m,n\in \N$, the \emph{$q$-analog of the Fibonomial number} is defined as
\begin{gather*}
\BigqFibonomial{m}{n} := \frac{\qFibofactorial{m+n}}{\qFibofactorial{m}\cdot \qFibofactorial{n}},
\end{gather*}
where
$\qFibofactorial{n}:= \prod\limits_{k=1}^n [F_k]_q$
is the $q$-Fibonacci analog of $n!$.
Surprisingly, this rational expression turns out to be a polynomial. Our objective is to present a combinatorial model to describe it.
In order to achieve this, we will introduce some $q$-weights associated to path-domino tilings of an $m\times n$ rectangle (with $m$ columns and $n$ rows).

Let $T\in \tilings{m}{n}$ be a path-domino tiling of an $m\times n$ rectangle.
The \emph{$q$-weights} of the possible tiles in $T$ are defined as follows

\begin{center}
\begin{tabular}{cccc}
 $\omega \left( \,
\begin{tikzpicture}[x=\boxsize,y=\boxsize,baseline=-4.5mm]
\monomino{0}{0}
\end{tikzpicture}
\, \right) = 1$,
& \
 $\omega \left( \,
\begin{tikzpicture}[x=\boxsize,y=\boxsize,baseline=-4.5mm]
\horizontalDomino{0}{0}
\end{tikzpicture}
\, \right) = q^{F_iF_j}$,
& \
 $\omega \left( \,
\begin{tikzpicture}[x=\boxsize,y=\boxsize,baseline=-8mm]
\verticalDomino{0}{0}
\end{tikzpicture}
\, \right) = q^{F_iF_j}$,
& \
 $\omega \left( \,
\begin{tikzpicture}[x=\boxsize,y=\boxsize,baseline=-8mm]
\specialDomino{0}{0}
\end{tikzpicture}
\, \right) = q^{F_{i+1}F_j}$,
\end{tabular}
\end{center}

\noindent where $(i,j)$ denotes the coordinate of the top-right corner of the tile, and the shaded
vertical domino represents a special vertical domino touching the path from below.
The \emph{$q$-weight} of~$T$ is defined as the product of the weights of its tiles; see an example in Fig.~\ref{fig_pathdominotiling}.
The following theorem is one of our main results.

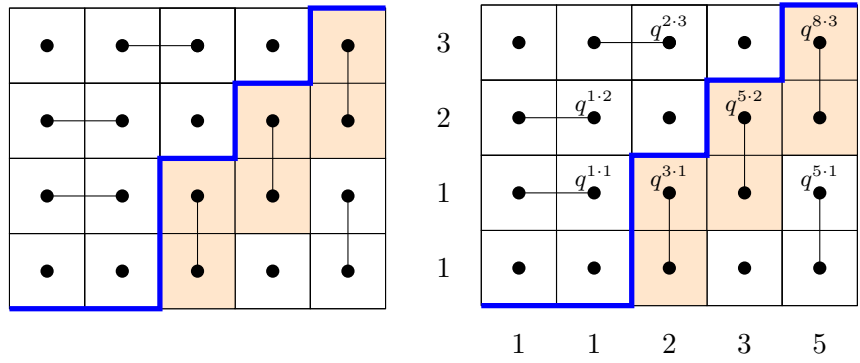
\begin{figure}[t]\centering
\begin{tikzpicture}
\node at (0.5,-0.5) {\color{white} a}; % this is just a trick to move the first figure so that both figures are at the same hegiht (I did not know how to do it differently)
%grid
\grid{5}{4}
%monominos
\monomino{1}{1}
\monomino{1}{4}
\monomino{2}{1}
\monomino{3}{3}
\monomino{3}{4}
\monomino{4}{4}
\monomino{4}{1}
%horizontal dominos:
\horizontalDomino{2}{2}
\horizontalDomino{2}{3}
\horizontalDomino{3}{4}
%vertical dominos:
\verticalDomino{5}{2}
%special vertical dominos:
\specialDomino{3}{2}
\specialDomino{4}{3}
\specialDomino{5}{4}
%lattice path:
\draw[line width=\pathwidth,color=blue]
(0,0)
-- ++(2,0)
-- ++(0,2)
-- ++(1,0)
-- ++(0,1)
-- ++(1,0)
-- ++(0,1)
-- ++(1,0);
\end{tikzpicture}
\quad
\begin{tikzpicture}
%grid
\labeledGrid{5}{4}{Fibonacci}
%monominos
\monomino{1}{1}
\monomino{1}{4}
\monomino{2}{1}
\monomino{3}{3}
\monomino{3}{4}
\monomino{4}{4}
\monomino{4}{1}
%horizontal dominos:
\labeledHorizontalDomino{3}{4}{Fibonacci}
\labeledHorizontalDomino{2}{3}{Fibonacci}
\labeledHorizontalDomino{2}{2}{Fibonacci}
%special vertical dominos:
\labeledSpecialDomino{3}{2}{Fibonacci}
\labeledSpecialDomino{4}{3}{Fibonacci}
\labeledSpecialDomino{5}{4}{Fibonacci}
%vertical dominos:
\labeledVerticalDomino{5}{2}{Fibonacci}
%lattice path:
\draw[line width=\pathwidth,color=blue]
(0,0)
-- ++(2,0)
-- ++(0,2)
-- ++(1,0)
-- ++(0,1)
-- ++(1,0)
-- ++(0,1)
-- ++(1,0);
%special vertical dominos:
\labeledSpecialDominoLabeling{3}{2}{Fibonacci}
\labeledSpecialDominoLabeling{4}{3}{Fibonacci}
\labeledSpecialDominoLabeling{5}{4}{Fibonacci}

\end{tikzpicture}
	\caption{A path-domino tiling $T$ of a $5\times 4$ rectangle (left), and the $q$-Fibonacci weights of its tiles (right).
		The weight of the tiling is the product of the weights of its tiles,
		${\omega(T)=q^{1+2+6+3+10+24+5}=q^{51}}$.}
	\label{fig_pathdominotiling}
\end{figure}

\begin{Theorem}\label{thm_qFibonomial}
	For $m,n\in \N$, the $q$-analog of the Fibonomial number is a polynomial in $q$ with non-negative integer coefficients. It
	can be computed as
	\[
	\BigqFibonomial{m}{n} = \sum_{T\in \tilings{m}{n}} \omega(T).
	\]
\end{Theorem}
\begin{proof}Let us start proving the result for the initial cases $m=1$ or $n=1$.

For $n=1$, we have $\qFibonomial{m}{1}=[F_{m+1}]_q$. The collection $\tilings{m}{1}$ coincides with the tilings of an $m$-strip with dominos and monominos, since only the last step of the specified lattice path can be a north step because of the special vertical domino condition.
	The weight of a domino in a~tiling, whose top-right corner has coordinate $(i,1)$, is $q^{F_iF_1}=q^{F_i}$.
	Therefore, the result follows from Lemma~\ref{lem_q_analog_Fibonacci}.
	
	For $m=1$, we have $\qFibonomial{1}{n}=[F_{n+1}]_q$.
	The collection $\tilings{1}{n}$ can be identified with the collection of tilings of a vertical $n$-strip with dominos and monominos, where the topmost domino has a~special weight.
	The weight of a usual vertical domino, whose top-right corner has coordinate~$(1,j)$, is $q^{F_1F_j}=q^{F_j}$, while
	the weight of a special vertical domino located at the same place is $q^{F_{1+1}F_j}=q^{F_j}$.
	Therefore, the result also follows from Lemma~\ref{lem_q_analog_Fibonacci}.
	
	Now assume the result holds when $m=1$ or $n=1$.
	Letting $m,n>1$ and using equation~\eqref{lem_qFibIdentities2} in the following equation we obtain
	\begin{align*}
	\BigqFibonomial{m}{n}
	&=
	\frac{[F_{m+n}]_q\qFibofactorial{m+n-1}}{\qFibofactorial{m}\cdot \qFibofactorial{n}} \\
	&=
	[F_{m+1}]_{q^{F_n}}\frac{\qFibofactorial{m+n-1}}{\qFibofactorial{m}\cdot \qFibofactorial{n-1}}
	+ q^{F_nF_{m+1}}[F_{n-1}]_{q^{F_m}} \frac{\qFibofactorial{m+n-1}}{\qFibofactorial{m-1}\cdot \qFibofactorial{n}} \\
	&=
	[F_{m+1}]_{q^{F_n}}\BigqFibonomial{m}{n-1}
	+ q^{F_nF_{m+1}}[F_{n-1}]_{q^{F_m}} \BigqFibonomial{m-1}{n}.
	\end{align*}
	By induction (and using again Lemma~\ref{lem_q_analog_Fibonacci}), the first term of the sum is
	the weighted counting of the path-domino tilings of the $m\times n$ rectangle whose specified path ends with a north step,
	while the second term is the weighted counting of those finishing with an east step. Indeed, the path-domino tilings whose path ends with a north step have an extra contribution $[F_{m+1}]_{q^{F_n}}$
	corresponding to the tilings of the last row with horizontal dominos and monominos. The path-domino tilings whose path
	ends with an east step have an extra contribution $q^{F_nF_{m+1}}[F_{n-1}]_{q^{F_m}}$; this corresponds to the weight of the forced special vertical domino ($q^{F_nF_{m+1}}$)
	and the tilings of the remaining $(n-2)$-strip in the last column ($[F_{n-1}]_{q^{F_m}}$).
\end{proof}

We have checked the following conjecture by computer for $m,n \leq 10$:
\begin{Conjecture} \label{conj:unimodal}
	The polynomials $\qFibonomial{m}{n}$ are unimodal.
\end{Conjecture}

\begin{figure}[t]	\centering
%\begin{subfigure}[b]{.15\linewidth}
\begin{tikzpicture}
\grid{2}{2}
\monomino{1}{1}
\monomino{2}{1}
\horizontalDomino{2}{2}
\draw[line width=\pathwidth,color=blue] (0,0) -- ++(2,0) -- ++(0,2);
\node at (1,-0.5) {$q$};
\end{tikzpicture}
%\end{subfigure}
\quad
%\begin{subfigure}[b]{.15\linewidth}
\begin{tikzpicture}
\grid{2}{2}
\monomino{1}{2}
\monomino{2}{2}
\horizontalDomino{2}{1}
\draw[line width=\pathwidth,color=blue] (0,0) -- ++(2,0) -- ++(0,2);
\node at (1,-0.5) {$q$};
\end{tikzpicture}
%\end{subfigure}
\quad
%\begin{subfigure}[b]{.15\linewidth}
\begin{tikzpicture}
\grid{2}{2}
\horizontalDomino{2}{1}
\horizontalDomino{2}{2}
\draw[line width=\pathwidth,color=blue] (0,0) -- ++(2,0) -- ++(0,2);
\node at (1,-0.5) {$q^2$};
\end{tikzpicture}
%\end{subfigure}
\quad
%\begin{subfigure}[b]{.15\linewidth}
\begin{tikzpicture}
\grid{2}{2}
\monomino{1}{1}
\monomino{1}{2}
\specialDomino{2}{2}
\draw[line width=\pathwidth,color=blue] (0,0) -- ++(1,0) -- ++(0,2) -- ++(1,0);
\node at (1,-0.5) {$q^2$};
\end{tikzpicture}
%\end{subfigure}
\quad
%\begin{subfigure}[b]{.15\linewidth}
\begin{tikzpicture}
\grid{2}{2}
\specialDomino{1}{2}
\specialDomino{2}{2}
\draw[line width=\pathwidth,color=blue] (0,0) -- ++(0,2) -- ++(2,0);
\node at (1,-0.5) {$q^3$};
\end{tikzpicture}
%\end{subfigure}
	\caption{The six path-domino tilings of a $2\times 2$ rectangle, and their contribution to the $q$-Fibonomial
		$\protect\qFibonomial{2}{2}=\frac{[F_4][F_3]}{[F_2][F_1]}=\frac{[3][2]}{[1][1]}=1+2q+2q^2+q^3$
		when computed as a generating function $\sum\limits_{T \in \tilings{2}{2}} \omega(T)$.}
	\label{fig_pathdominotilings22}
\end{figure}
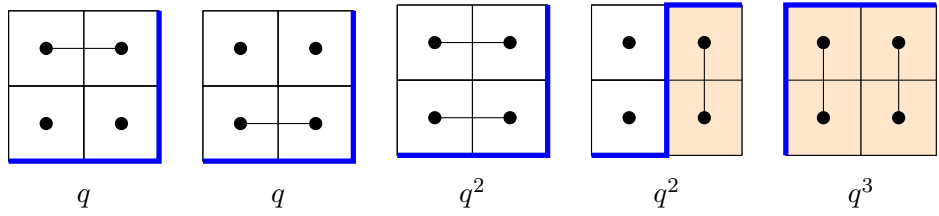

\begin{Example}
	Fig.~\ref{fig_pathdominotilings22} illustrates an example of the $q$-Fibonomial for $m=n=2$.
\end{Example}

\begin{Example}[$n=2$]\label{ex_qFibSpiral}
	Let $m\in \N$ and $n=2$. Theorem~\ref{thm_qFibonomial} leads to the identity
	\begin{gather} \label{identity_qFibSpiral}
	[F_{m+2}][F_{m+1}] = \sum_{k=1}^{m+1} q^{c_k^m} [F_k]^2,
	\end{gather}
	where
	$
	c_k^m=\sum\limits_{i=k}^m F_{i+1}$.

	The left hand side comes from the equality $\qFibonomial{m}{2}=[F_{m+2}][F_{m+1}]$.
	The right hand side is the sum of the weights of all path-domino tilings of an $m\times 2$ rectangle.
	In fact, the term $q^{c_k^m} [F_k]^2$ indicates the sum of the weights of the path-domino tilings whose
	specified path is $E^{k-1}N^2E^{m-(k-1)}$: $q^{c_k^m}$ is the product of the weights of the special vertical dominos, and $[F_k]^2$
	is the weight of the two horizontal rows above the path. Since there are no more possibilities for the specified path due to the special vertical domino condition,
	the identity \eqref{identity_qFibSpiral} follows.
	The evaluation at $q=1$ recovers
	\begin{align} \label{identity_FibSpiral}
	F_{m+2}F_{m+1} = \sum_{k=1}^{m+1} F_k^2.
	\end{align}	
\end{Example}
\begin{Remark}
	Equation~\eqref{identity_FibSpiral} is a well known identity due to its relation with the golden ratio and golden spirals in nature, see for instance~\cite{goldenRationFibonacciBook,fabulousFibonacciBook}. The left hand side of the equation is the area of a $F_{m+2}\times F_{m+1}$ rectangle, which can be subdivided into a sequence of squares, with side lengths $F_1,F_2,\dots , F_{m+1}$, forming a spiral as illustrated in Fig.~\ref{fig_FibonacciSpiral}~(left). This Fibonacci spiral is an approximation of the golden spiral, a special case of logarithmic spirals which describe the shape of various natural phenomena such as galaxies, nautilus shells and hurricanes. On the other hand, equation~\eqref{identity_qFibSpiral} also has a natural geometric interpretation. The left hand side represents the weighted area of a $F_{m+2}\times F_{m+1}$ rectangle, where a unit square whose bottom-left corner is located at $(i,j)$ has weight $q^{i+j}$. This rectangle can be subdivided into a sequence of squares, with side lengths $F_1,F_2,\dots , F_{m+1}$, in the north-east direction as illustrated in Fig.~\ref{fig_FibonacciSpiral}~(right). The sum of their weighted areas is exactly the right hand side of equation~\eqref{identity_qFibSpiral}. This sum can also be interpreted as the ``mass'' of the rectangle, where the $F_k$-square has density $d(F_k)=q^{c_k^m}$. This density increases according to the ratio $\frac{d(F_k)}{d(F_{k+1})}=q^{F_{k+1}}$, satisfying the initial condition $d(F_{m+1})=1$.
	It would be interesting to assign these densities to the squares giving rise to the Fibonacci spiral on Fig.~\ref{fig_FibonacciSpiral}~(left), and see if the resulting equation has some physical meaning. Or even more interesting, to have a continuous version of the equation representing the mass of the Fibonacci spiral (or golden spiral), in order to describe some physical phenomenon in nature. For instance, it is quite natural to think that the density of galaxies grows exponentially as it approaches the center of the spiral. In Example~\ref{ex_ellFibSpiral}, we additionally provide a generalization of equation~\eqref{identity_qFibSpiral} using elliptic weight functions.
\end{Remark}	

\begin{figure}[t]	\centering
\begin{tikzpicture}[scale=0.3]
%\draw[dotted,thin, lightgray] (-6,-3) grid (15,10);
%squares
\fiboSquare{-1}{0}{1}
\fiboSquare{-1}{1}{1}
\fiboSquare{0}{0}{2}
\fiboSquare{-1}{-3}{3}
\fiboSquare{-6}{-3}{5}
\fiboSquare{-6}{2}{8}
\fiboSquare{2}{-3}{13}
%spiral approximation
\draw[color=blue, thick] (0,0) arc[start angle=270, end angle=180, radius=1] arc[start angle=180, end angle=90, radius=1] arc[start angle=90, end angle=0, radius=2] arc[start angle=0, end angle=-90, radius=3] arc[start angle=270, end angle=180, radius=5] arc[start angle=180, end angle=90, radius=8] arc[start angle=90, end angle=0, radius=13];

\end{tikzpicture} \quad
\begin{tikzpicture}[scale=0.3]
	\draw[thin, lightgray] (-6,-3) grid (15,10);
	\fiboSquare{-6}{-3}{13}
	\fiboSquare{7}{-3}{8}
	\fiboSquare{7}{5}{5}
	\fiboSquare{12}{5}{3}
	\fiboSquare{12}{8}{2}
	\fiboSquare{14}{8}{1}
	\fiboSquare{14}{9}{1}
\end{tikzpicture}
	\caption{Geometric interpretation of equations~\eqref{identity_qFibSpiral} and~\eqref{identity_FibSpiral}.}	\label{fig_FibonacciSpiral}
\end{figure}
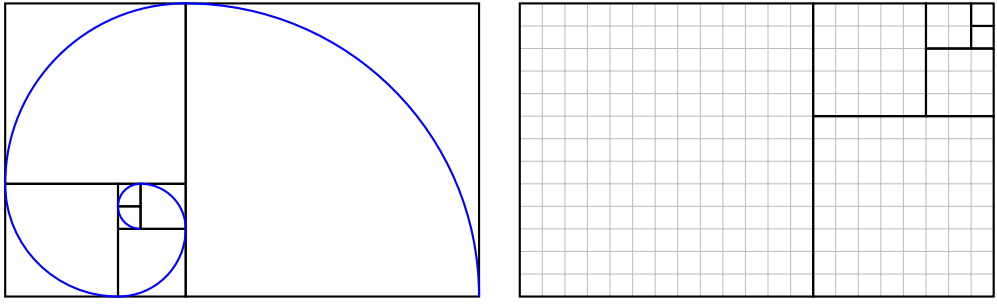

\section{Elliptic analog of the Fibonomial numbers}\label{sec_elliptic}

As in the previous section, the elliptic analog of the Fibonomial number is obtained
by replacing each term in the binomial coefficient by its corresponding Fibonacci ``elliptic number''.
The elliptic number used here is a slight modification of the elliptic number introduced by Schlosser and Yoo in~\cite{elliptic_schlosser_2017}, motivated by work of Schlosser on
elliptic binomial coefficients~\cite{noncommutative_schlosser_2011}.
The elliptic number is an elliptic function that generalizes the $q$-analog of a number, and plays an important role in the theory of hypergeometric series and special functions. Over the last years, elliptic enumeration of combinatorial models and elliptic generalization of special combinatorial numbers is becoming an active research area due to its natural appearance in the theory of hypergeometric functions, see, e.g., \cite{q_borodin_2010, elliptic_schlosser_2007, elliptic_schlosser_2017}.
In this section, we contribute to this general program by studying an elliptic generalization of the Fibonomial number.
Before defining the elliptic analog, we introduce some basic notations and definitions concerning elliptic functions.

An \emph{elliptic function} is a function defined over the complex numbers that is meromorphic and doubly periodic.
It is well known (cf., e.g., \cite{elliptic_rosengren_2016,elliptische_weber_1891}) that elliptic functions can be obtained as quotients of \emph{modified Jacobi theta functions}.
These are defined as
\[
\theta(x;p) := \prod_{j\geq 0}
\left(
\big(1-p^jx\big)\left(1-\frac{p^{j+1}}{x}\right)
\right),
\qquad
\theta(x_1,\dots, x_\ell;p) =
\prod_{k=1}^\ell \theta(x_k;p),
\]
where $x,x_1,\dots,x_\ell \neq 0$ and $|p|<1$.
The \emph{elliptic analog} of a natural number $n\in \N$ (or simply \emph{elliptic number}) is defined as
\begin{gather*}
\elliptic{n} :=
\frac
{\theta\big(q^n,aq^n,bq,\frac{a}{b}q;p\big)}
{\theta\big(q,aq,bq^{n},\frac{a}{b}q^{n};p\big)}.
\end{gather*}
Here, $a$, $b$ are two additional parameters. This definition follows the definition in \cite{logconcav_schlosser_2020} and corresponds to the definition in \cite{elliptic_schlosser_2017} subject to the substitution $b \mapsto bq^{-1}$.\footnote{We are grateful to an anonymous referee for suggesting this substitution. This edition made many of our formulas look simpler and more symmetric.} The elliptic number is indeed an elliptic function in its parameters~\cite[Remark~4]{elliptic_schlosser_2017}.
Taking the limit $p\rightarrow 0$, then $a\rightarrow 0$ and then $b\rightarrow 0$, one recovers the $q$-analog $[n]_q$. We simply define the \emph{elliptic analog of the Fibonacci number} $F_n$ by $[F_n]_{a,b;q,p}$.

For $m,n\in \N$, the \emph{elliptic analog of the Fibonomial number} is defined as
\begin{gather*}
\BigellipticFibonomial{m}{n} := \frac{\ellipticFiboFactorial{m+n}}{\ellipticFiboFactorial{m}\cdot \ellipticFiboFactorial{n}},
\end{gather*}
where
$\ellipticFiboFactorial{n}:= \prod\limits_{k=1}^n \elliptic{F_k}$
is the elliptic Fibonacci analog of $n!$.

Similarly as before, the elliptic Fibonomial number counts path-domino tilings of an $m\times n$ rectangle
according to certain elliptic weights.
For $T\in \tilings{m}{n}$, the \emph{elliptic weights} of the possible tiles in $T$ are defined as follows
\begin{center}
 $\eweight \left( \,
\begin{tikzpicture}[x=\boxsize,y=\boxsize,baseline=-4.5mm]
\monomino{0}{0}
\end{tikzpicture}
\, \right) = 1$,
\quad
 $\eweight \left( \,
\begin{tikzpicture}[x=\boxsize,y=\boxsize,baseline=-4.5mm]
\horizontalDomino{0}{0}
\end{tikzpicture}
\, \right) = \omega_1(i,j)$,
\quad
 $\eweight \left( \,
\begin{tikzpicture}[x=\boxsize,y=\boxsize,baseline=-8mm]
\verticalDomino{0}{0}
\end{tikzpicture}
\, \right) = \omega_1(j,i)$,
\quad
 $\eweight \left( \,
\begin{tikzpicture}[x=\boxsize,y=\boxsize,baseline=-8mm]
\specialDomino{0}{0}
\end{tikzpicture}
\, \right) = \omega_2(i,j)$,
\end{center}

\noindent
where $(i,j)$ denotes the coordinate of the top-right corner of the tile, the shaded
vertical domino represents a special vertical domino touching the path from below, and
\begin{gather*}
\omega_1(i,j) := v_{a,b;q^{F_j},p}(F_i,F_{i-1}),\\ % \label{def_ellweights1} \\
\omega_2(i,j) := v_{a,b;q,p}(F_{i+1}F_{j},F_{i}F_{j-1}) % \label{def_ellweights2}
\end{gather*}
are defined in terms of the following expression
\begin{gather*}
v_{a,b;q,p}(m,n) := \frac{\theta\big(aq^{2m+n},b,bq^{n},\frac{a}{b}q^{n},\frac{a}{b};p\big)} {\theta\big(aq^{n},bq^{m},bq^{m+n},\frac{a}{b}q^{m},\frac{a}{b}q^{m+n};p\big)}q^m.
\end{gather*}

The elliptic weight function $v_{a,b;q,p}(m,n)$ comes from equation~\eqref{lem_ellidentities1} below and is the missing factor in that equation. Note that the weight of a ``regular'' vertical domino is evaluated at $(j,i)$ instead of $(i,j)$. This transposition
does not make any difference for the $q$-analog of the Fibonomial numbers, but it does for the elliptic case.
The \emph{elliptic weight~$\eweight(T)$} of~$T$ is defined as the product of the weights of its tiles; see an example in Fig.~\ref{fig_tilesellipticweights}. The elliptic weight is a~generalization of the $q$-weight, since we obtain the $q$-weight by taking the limit $p\rightarrow 0$, $a\rightarrow 0$ and $b\rightarrow 0$ in this order.

\begin{figure}[t]	\centering
\begin{tikzpicture}[scale=1.1]
%grid
\ordinaryLabeledGrid{5}{4}

%monominos
\monomino{1}{1}
\monomino{1}{4}
\monomino{2}{1}
\monomino{3}{3}
\monomino{3}{4}
\monomino{4}{4}
\monomino{4}{1}
%horizontal dominos:
\ellipticLabeledHorizontalDomino{3}{4}
\ellipticLabeledHorizontalDomino{2}{3}
\ellipticLabeledHorizontalDomino{2}{2}
%special vertical dominos:
\ellipticLabeledSpecialDomino{5}{4}
\ellipticLabeledSpecialDomino{4}{3}
\ellipticLabeledSpecialDomino{3}{2}
%vertical dominos:
\ellipticLabeledVerticalDomino{5}{2}
%lattice path:
\draw[line width=\pathwidth,color=blue]
(0,0)
-- ++(2,0)
-- ++(0,2)
-- ++(1,0)
-- ++(0,1)
-- ++(1,0)
-- ++(0,1)
-- ++(1,0);
%special vertical dominos labeling:
\ellipticLabeledSpecialDominoLabeling{5}{4}
\ellipticLabeledSpecialDominoLabeling{4}{3}
\ellipticLabeledSpecialDominoLabeling{3}{2}

\end{tikzpicture}

	\caption{The path-domino tiling $T$ from Fig.~\ref{fig_pathdominotiling} with the elliptic Fibonacci weights of its tiles.
		The weight of the tiling is the product of the weights of its tiles.}\label{fig_tilesellipticweights}
\end{figure}
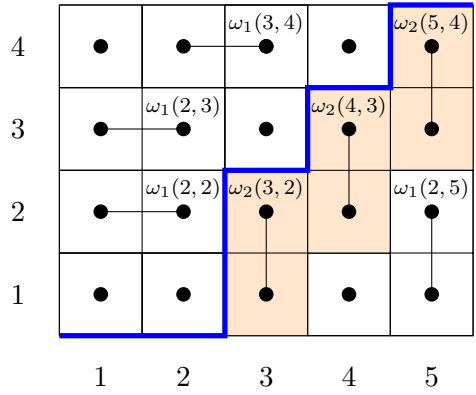

\begin{Theorem}\label{thm_ellipticFibonomial}
	For $m,n\in \N$, the elliptic analog of the Fibonomial number can be computed as
	\[
	\BigellipticFibonomial{m}{n} = \sum_{T\in \tilings{m}{n}} \eweight(T).
	\]
\end{Theorem}

The proof of this theorem follows the same steps as the proof of Theorem~\ref{thm_qFibonomial}.
The proofs of the technical lemmas and examples use some basic properties of theta functions summarized in the following proposition, which are essential in the theory of elliptic hypergeometric series.

\begin{Proposition}[{cf.\ \cite[p.~451, Example~5]{elliptische_weber_1891}}]
	\label{prop_thetaBasicProperties}
	The theta function satisfies the following basic properties
	\begin{gather}
	\theta(x;0) = 1-x, \nonumber\\ %\label{lem_thetaBasicProperties0} \\
	\theta\big(\tfrac{1}{x};p\big) = -\tfrac{1}{x} \theta(x;p), \nonumber\\
	\theta(px;p) = -\tfrac{1}{x} \theta(x;p),\nonumber \\
	\theta\big(xy,\tfrac{x}{y},uz,\tfrac{u}{z};p\big)=
	\theta\big(uy,\tfrac{u}{y},xz,\tfrac{x}{z};p\big)+
	\tfrac{x}{z}
	\theta\big(zy,\tfrac{z}{y},ux,\tfrac{u}{x};p\big). \label{lem_thetaBasicProperties2}
	\end{gather}
\end{Proposition}

Before proving Theorem \ref{thm_ellipticFibonomial}, let us again prove some straightforward lemmas:
\begin{Lemma} \label{lem_ellidentities}
	For $m,n\in \N$, the following identities hold
	\begin{gather}
	[m+n]_{a,b;q,p} = [m]_{a,b;q,p} + v_{a,b;q,p}(m,n) [n]_{a,b;q,p}, \label{lem_ellidentities1} \\
	[m\cdot n]_{a,b;q,p} = [m]_{a,b;q,p} [n]_{a,b;q^m,p}. \label{lem_ellidentities2}
	\end{gather}
\end{Lemma}
\begin{proof}
	Rearranging the left hand side of equation~\eqref{lem_ellidentities1} and applying equation~\eqref{lem_thetaBasicProperties2} (by taking $x=a^{\frac{1}{2}}q^m$, $y=a^{\frac{1}{2}}b^{-1}$, $u=a^{\frac{1}{2}}q^{m+n}$ and $z=a^{\frac{1}{2}}$) yields
	\begin{gather*}
	[m+n]_{a,b;q,p} =\frac{
		\theta\big(bq,\frac{a}{b}q;p\big)
	}
	{
		\theta\big(q,aq,bq^{m+n},\frac{a}{b}q^{m+n};p\big)
	}
	\cdot	
	\frac{
		\theta\big(\frac{a}{b}q^{m}, bq^{m},aq^{m+n},q^{m+n};p\big)
	}
	{
		\theta\big(\frac{a}{b}q^{m}, bq^{m};p\big)
	}	\\
\hphantom{[m+n]_{a,b;q,p}}{} =		
	\frac{
		\theta\big(bq,\frac{a}{b}q;p\big)}
	{
		\theta\big(q,aq,bq^{m+n},\frac{a}{b}q^{m+n};p\big)
	}
	\\
\hphantom{[m+n]_{a,b;q,p}=}{} \times
	\frac{\theta\big(\frac{a}{b}q^{m+n},bq^{m+n},aq^m,q^m;p\big)+q^m\theta\big(\frac{a}{b},b, aq^{2m+n}, q^n;p\big)}
	{
		\theta(\frac{a}{b}q^{m}, bq^{m};p)
	} \\
\hphantom{[m+n]_{a,b;q,p}}{}=[m]_{a,b;q,p} + \frac{\theta\big(aq^{2m+n},b,bq^{n},\frac{a}{b}q^{n},\frac{a}{b};p\big)} {\theta\big(aq^{n},bq^{m},bq^{m+n},\frac{a}{b}q^{m},\frac{a}{b}q^{m+n};p\big)}q^m [n]_{a,b;q,p}.
	\end{gather*}
	Equation \eqref{lem_ellidentities2} follows from simple cancellations.
\end{proof}

Using the same arguments as in the $q$-case and replacing the weight of a domino whose top-right coordinate is $(i,1)$ by $\omega_1 (i,1)$,
we obtain the following lemmas.
\begin{Lemma}\label{lem_ellanalog_Fibonacci}
	For $n\in \N$, the elliptic analog of the Fibonacci numbers can be computed as
	\[
	[F_n]_{a,b;q,p} = \sum_T {\eweight(T)},
	\]
	where the sum ranges over all tilings of an $(n-1)$-strip using dominos and monominos.
\end{Lemma}
\begin{proof}
	The result is clearly true for $n=1,2$. Let $n>2$, applying equation~\eqref{lem_ellidentities1} from Lemma~\ref{lem_ellidentities} we obtain
	\begin{align*}
	[F_n]_{a,b;q,p} &= [F_{n-1}+F_{n-2}]_{a,b;q,p} = [F_{n-1}]_{a,b;q,p} +
	v_{a,b;q,p}(F_{n-1},F_{n-2})
	[F_{n-2}]_{a,b;q,p}\\
	&= [F_{n-1}]_{a,b;q,p} +	\omega_1(n-1,1)	[F_{n-2}]_{a,b;q,p}.
	\end{align*}
	By induction, the first term of this sum corresponds to the tilings of an $(n-1)$-strip that finish with a monomino,
	while the second term to the tilings of an $(n-1)$-strip that finish with a~domino.
\end{proof}

\begin{Lemma}\label{lem_ellFibIdentities}
	For $m,n\in \N$, the following identity holds
	\begin{align}
	[F_{m+n}]_{a,b;q,p} & = [F_n]_{a,b;q,p} [F_{m+1}]_{a,b;q^{F_n},p} + \omega_2 (m,n) [F_m]_{a,b;q,p} [F_{n-1}]_{a,b;q^{F_m},p}. \label{lem_ellFibIdentities1}
	\end{align}
\end{Lemma}
\begin{proof}Applying Lemma~\ref{lem_ellidentities} to equation~\eqref{lem_qFibIdentities1} leads to equation~\eqref{lem_ellFibIdentities1}.
\end{proof}

Now, we have all tools to prove the main theorem in this section.

\begin{proof}[Proof of Theorem~\ref{thm_ellipticFibonomial}]
	For $n=1$, we have $\ellipticFibonomial{m}{1}=[F_{m+1}]_{a,b;q,p}$. The collection $\tilings{m}{1}$ coincides with the tilings of an $m$-strip with dominos and monominos,
	since only the last step of the specified lattice path can be a north step because of the special vertical domino condition.
	Therefore, the result follows from Lemma~\ref{lem_ellanalog_Fibonacci}.
	
	For $m=1$, we have $\ellipticFibonomial{1}{n}=[F_{n+1}]_{a,b;q,p}$.
	The collection $\tilings{1}{n}$ can be identified with the collection of tilings of a vertical $n$-strip with dominos and monominos, where the topmost domino has a special weight.
	The weight of a usual vertical domino, whose top-right corner has coordinate $(1,j)$, is $\omega_1 (j,1)$ while
	the weight of a special vertical domino located at the same place is $\omega_2 (1,j)$. Since $\omega_2 (1,j)=\omega_1 (j,1)$, the result also follows from Lemma~\ref{lem_ellanalog_Fibonacci}.
	
	Now assume the result holds when $m=1$ or $n=1$.
	Letting $m,n>1$ and using equation~\eqref{lem_ellFibIdentities1} in the following equation we obtain
	\begin{align*}
	\BigellipticFibonomial{m}{n}
	&=
	\frac{[F_{m+n}]_{a,b;q,p}\ellipticFiboFactorial{m+n-1}}{\ellipticFiboFactorial{m}\cdot \ellipticFiboFactorial{n}}
=
	[F_{m+1}]_{a,b;q^{F_n},p}\BigellipticFibonomial{m}{n-1} \\
	&\quad+ \omega_2 (m,n) [F_{n-1}]_{a,b;q^{F_m},p} \BigellipticFibonomial{m-1}{n}.
	\end{align*}
	By induction (and using again Lemma~\ref{lem_ellanalog_Fibonacci}), the first term of the sum is
	the weighted counting of the path-domino tilings of the $m\times n$ rectangle whose specified path ends with a north step,
	while the second term is the weighted counting of those finishing with an east step.
	
	Indeed, the path-domino tilings whose path ends with a north step have an extra contribution $[F_{m+1}]_{a,b;q^{F_n},p}$.
	This corresponds to the weighted enumeration of the tilings of the last row with horizontal dominos and monominos.
	This follows from the fact that both quantities satisfy the same initial conditions and recurrence relation, which is obtained by applying equation~\eqref{lem_ellidentities1} to $F_m+F_{m-1}$
	\[
	[F_{m+1}]_{a,b;q^{F_n},p} =
	[F_{m}]_{a,b;q^{F_n},p} + \omega_1(m,n)
	[F_{m-1}]_{a,b;q^{F_n},p}.
	\]
	The path-domino tilings whose path
	ends with an east step have an extra contribution
	\[
	\omega_2 (m,n)[F_{n-1}]_{a,b;q^{F_m},p}.
	\]
This corresponds to the weight of the forced special vertical domino ($\omega_2 (m,n)$)
	and the tilings of the remaining $(n-2)$-strip in the last column ($[F_{n-1}]_{a,b;q^{F_m},p}$).
\end{proof}

\begin{Example}[$n=2$]\label{ex_ellFibSpiral}
	The identity \eqref{identity_qFibSpiral} in Example~\ref{ex_qFibSpiral} for $m\in \N$ and $n=2$ generalizes in the elliptic case to
\begin{gather*} %\label{identity_ellFibSpiral}
	\elliptic{F_{m+2}}\elliptic{F_{m+1}} = \sum_{k=1}^{m+1} \Omega_k^m \elliptic{F_k}^2,
\end{gather*}
where $\Omega_k^m=\prod\limits_{i=k}^m \omega_2 (i,2)$ is the product of the weights of the special vertical dominos, and $\elliptic{F_k}^2=\elliptic{F_k} \cdot [F_k]_{a,b;q^{F_2},p}$ is the weight of the two horizontal rows above the path.
\end{Example}
\begin{Example}[$a,b;p \to 0$]
	By computing the limits $p \to 0$, $a \to 0$ and $b \to 0$ (in this order) of~$\eweight(T)$ and $\ellipticFibonomial{m}{n}$ we obtain Theorem~\ref{thm_qFibonomial}.
\end{Example}

\begin{Remark}
	The elliptic analog of the Fibonomial number is a Fibonacci analog of the ``regular'' elliptic binomial coefficient
	\begin{gather*}
	\qBinomial{n}{k}{a,b;q,p} := \frac{[n]_{a,b;q,p}!}{[k]_{a,b;q,p}! \cdot [n-k]_{a,b;q,p}!},
	\end{gather*}
	where $[n]_{a,b;q,p}!=\prod\limits_{i=1}^{n} [i]_{a,b;q,p}$.
	This can be expressed as
	\begin{gather}\label{eq_ellipticbinomial}
	\qBinomial{n}{k}{a,b;q,p} = \frac{(q^{n-k+1},aq^{n-k+1},bq,\frac{a}{b}q;q,p)_k}
	{(q,aq,bq^{n-k+1},\frac{a}{b}q^{n-k+1};q,p)_k},
	\end{gather}
	where the theta shifted factorial is defined as $(a;q,p)_k = \prod\limits_{i=0}^{k-1} \theta(aq^i;p)$ for $k>0$, $(a;q,p)_0=1$ and $(a_1,a_2,\dots,a_l;q,p)_k=\prod\limits_{i=1}^{l} (a_i;q,p)_k$.
	
	Several elliptic versions of the binomial coefficient have been already considered in the literature, see for instance~\cite{symmetric_rains_2006} or \cite{elliptic_schlosser_2007}.
	By substituting $b \mapsto bq^k$ in equation~\eqref{eq_ellipticbinomial}, we obtain the elliptic binomial coefficient defined by Schlosser in \cite{elliptic_schlosser_2007}
	\[
	\qBinomial{n}{k}{a,b;q,p}' := \frac{\big(q^{n-k+1},aq^{n-k+1},bq^{1+k},\frac{a}{b}q^{1-k};q,p\big)_k}{\big(q,aq,bq^{n+1},\frac{a}{b}q^{n-2k+1};q,p\big)_k}.
	\]
	One can also check that our regular elliptic binomial coefficient is different to one defined by Rains in~\cite[Definition~11]{symmetric_rains_2006}, even in the simplest case where both partitions are single columns of respective lengths $n$ and~$k$, cf.\ \cite[equation~(3.12)]{elliptic_schlosser_2017}.	
\end{Remark}
An interesting task for future works on the elliptic Fibonomials could be to generalize Fibonomial identities to ellitpic identities, see, e.g., \cite{combinatorial_benjamin_2014, combinatorial_reiland} for potential candidates.
For instance, we generalize a proof by Reiland \cite{combinatorial_reiland} for the identity
\[
\Fibonomial{m}{n}=\sum_{j=0}^{n} F_{m+1}^j F_{n-j-1} \Fibonomial{m-1}{n-j}
\]
in the following corollary.
\begin{Corollary} \label{cor_identity}
	For $m,n \geq 1$, the elliptic Fibonomial $\ellipticFibonomial{m}{n}$ is equal to
	\[
	\sum_{j=0}^{n} \left( \prod_{i=0}^{j-1} [F_{m+1}]_{a,b;q^{F_{n-i}},p} \right) [F_{n-1-j}]_{a,b;q^{F_{m}},p} w_2(m,n-j) \BigellipticFibonomial{m-1}{n-j}.
	\]	
\end{Corollary}
\begin{proof}For $0 \leq j \leq n$, consider all weighted path-domino tilings $\tilings{m}{n}^{j}$ of an $m \times n$ rectangle for $m,n \geq 1$ where the last horizontal step of the path ends at position $(m,n-j)$. To obtain the sum of the weights $\sum\limits_{T \in \tilings{m}{n}^{j}} \eweight(T)$, we subdivide the tilings (as illustrated in Fig.~\ref{fig_identitytiling}): The collection~$\tilings{m}{n}^{j}$ can be identified with the collection of
	\begin{itemize}\itemsep=0pt
		\item a path domino tiling of an $(m-1) \times (n-j)$ rectangle which adds weight $\ellipticFibonomial{m-1}{n-j}$,
		\item $j$ full rows on top which add weights $[F_{m+1}]_{a,b;q^{F_{n-i}},p}$ for $0 \leq i \leq j-1$ or $1$ if $j=0$,
		\item the tile below the last horizontal step, which is forced to be a special vertical domino with weight $\omega_2(m,n-j)$,
		\item and the column below the special vertical domino which adds weight $[F_{n-1-j}]_{a,b;q^{F_{m}},p}$
	\end{itemize}
to the sum. 	Summing over all $j$, we obtain the required convolution formula.
	Note that in the case $j=n-1$ the number of tilings is 0 since there has to be a vertical domino below the last horizontal step and indeed, there appears the factor $[F_0]_{a,b;q^{F_{m}},p}=0$ in the corresponding summand. For $j=n$ there is no special vertical domino and no column below the last east step of the path but $\omega_2(m,0)~=~1$ and $[F_{-1}]_{a,b;q^{F_{m}},p}=1$ (if we define $F_{-1}$ by $F_{-1}+F_0=F_1$ to be~$1$).
\end{proof}

\begin{figure}[t]\centering
\begin{tikzpicture}[scale=0.9]
%grid
\draw[draw=black] (0,0) rectangle ++(6,7);
\draw[dashed] (0,4) -- (6,4);
\draw[dashed] (0,5) -- (6,5);
\draw[dashed] (0,6) -- (6,6);
\draw[dashed] (5,0) -- (5,2);

%monominos

%horizontal dominos:

%special vertical dominos:
\ellipticLabeledSpecialDomino{6}{4}

%vertical dominos:
%lattice path:
\draw[line width=\pathwidth,color=blue]
(5,4)
-- ++(1,0)
-- ++(0,3);
%special vertical dominos labeling:
\node[right, inner sep=1.5pt] at (0.2,6.5) {$[F_{m+1}]_{a,b;q^{F_{n}},p}$};
\node[right, inner sep=1.5pt] at (0.2,5.5) {$[F_{m+1}]_{a,b;q^{F_{n-1}},p}$};
\node[right, inner sep=1.5pt] at (0.2,4.5) {$[F_{m+1}]_{a,b;q^{F_{n-(j-1)}},p}$};
\node[right, inner sep=1.5pt] at (6,3) {$\omega_2 (m,n-j)$};
\node[right, inner sep=1.5pt] at (6,1) {$[F_{n-1-j}]_{a,b;q^{F_m},p}$};

\node[inner sep=1.5pt] at (2.5,2) {$\BigellipticFibonomial{m-1}{n-j}$};
\end{tikzpicture}
	\caption{The parts of one summand of the right hand side of Corollary~\ref{cor_identity} for $m=6$, $n=7$ and $j=3$.}
	\label{fig_identitytiling}
\end{figure}
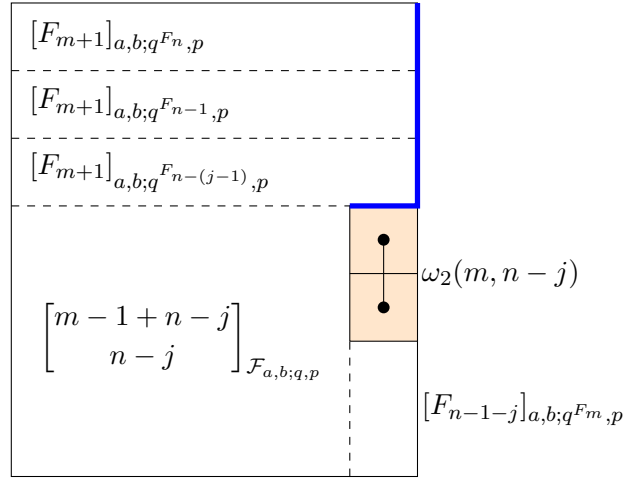

\section[The $q$-Fibonacci analog of the rational Catalan numbers]{The $\boldsymbol{q}$-Fibonacci analog of the rational Catalan numbers}
Given a pair of relatively prime numbers $m,n\in \N$ (that is, such that their greatest common divisor is~$(m,n):= \gcd\{m,n\}= 1$),
the \emph{$m,n$-Catalan number} is defined as
\[
\mnCat{m}{n} := \mnCatalan{m}{n}.
\]
This number is equal to the number of lattice paths from $(0,0)$ to $(m,n)$ that stay weakly above the main diagonal of the $m\times n$ rectangle.
The study of these numbers (also in the non-coprime case), and their $q$-analog and $q,t$-analog generalizations, is connected to numerous relevant topics including rectangular diagonal harmonics and Macdonald polynomials~\cite{aval_anote_2018,bergeron_open_2017,bergeron_someremarkable_2016,bergeron_compositional_2016,mellit_toric_2016}, Shi hyperplane arrangements and
affine Weyl groups~\cite{athanasiadis_refinement_2005,gorsky_affine_2016,thiel_anderson_2016}, affine Springer fibers~\cite{gorsky_affine_2016,hikita_affinespringer_2014}, affine Hecke algebras~\cite{cherednik_double_2005}, knot theory~\cite{gorsky_refinedknot_2015,gorsky_torus_2014}, and representation theory of Cherednik algebras~\cite{etingof_symplectic_2002,etingof_representations_2015,etingof_lecturenotes_2010}.

The \emph{$q$-Fibonacci analog of the $m,n$-Catalan number} is defined as
\[
[\mnFCat{m}{n}] :=
\BigmnqFiboCatalan{m}{n} = \frac{\qFibofactorial{m+n-1}}{\qFibofactorial{m}\qFibofactorial{n}}.
\]
Surprisingly, this rational expression also turns out to be a polynomial when the greatest common divisor $(m,n) \in \{1,2\}$.
Before proving this, we need the following lemma.
\begin{Lemma}[{\cite{divisibility_hoggatt_1974}}]
	For $m,n\in \N$, we have $(F_m,F_n)=F_{(m,n)}$.
\end{Lemma}

\begin{Proposition}[S.X.~Li \cite{ShuXiao_2015}]\label{prop_mnCatPolynomiality}
	If the greatest common divisor $(m,n)\in \{1,2\}$, then $[\mnFCat{m}{n}]$ is a polynomial in $q$ with integer coefficients.
\end{Proposition}

\begin{proof}
	First note that
	\[
	[F_n] \BigqFibonomial{m}{n} = [F_{m+n}] \BigqFibonomial{m}{n-1}.
	\]
	Since all the terms involved in this identity are polynomials, we know $[F_{m+n}]$ divides $[F_n] \qFibonomial{m}{n}$.
	But
	\[
	(F_{m+n},F_n) = F_{(m+n,n)} = F_{(m,n)} =1
	\]
	whenever $(m,n)$ is equal to 1 or 2. Since the roots of $[n]$ are the $n$-th roots of unity that are different to $1$, the polynomials $[F_{m+n}]$ and $[F_n]$ have no roots in common.
	Thus, $[F_{m+n}]$ divides $\qFibonomial{m}{n}$.
\end{proof}

Computational experimentation for $m,n \leq 15$ suggests that the coefficients of these polynomials are non-negative integers.
However, we do not have a proof nor a combinatorial model to describe them.

\begin{Openproblem}	\label{prob_qFib_rational}
	Find a combinatorial interpretation of the $q$-Fibonacci analog of the rational Catalan numbers.
\end{Openproblem}

\begin{Remark}
The model of Sagan and Savage in~\cite{combinatorial_sagan_2010} gives a combinatorial interpretation of the Lucas analog of the binomial coefficients. The \emph{Lucas polynomials} $\{n\}$ generalize the Fibonacci numbers, and are defined by the initial conditions $\{0\}=0$, $\{1\}=1$ and the recurrence $\{n\}=s\{n-1\}+t\{n-2\}$ for some variables $s$,~$t$. It was proven in \cite[Section~6.2]{combinatorial_bennett_2018}, that the Lucas analog of the rational $m,n$-Catalan numbers is also polynomial in $s$, $t$ with integer coefficients.\footnote{Their proof is reproduced from the proof of Proposition~\ref{prop_mnCatPolynomiality} with our permission.} However, to the best of our knowledge, no $q$-analogs of the Lucas-binomial coefficients have been studied in the literature. The proof of Proposition~\ref{prop_mnCatPolynomiality} was originally found by our colleague S.X.~Li, during discussions about this topic in the Algebraic Combinatorics Seminar at the Fields Institute in 2015. At this Seminar, our colleague F.~Aliniaeifard showed that Conjecture~\ref{conj:unimodal} implies the positivity of the coefficients.
\end{Remark}

\begin{Remark}Given a crystallographic Coxeter group $W$ with Coxeter exponents $e_1<e_2<\dots<e_n$, the rational $W$-Catalan number is defined as $ C_W(a) = \prod\limits_{i=1}^n \frac{a+e_i}{e_i+1}$, and this is an integer when $a$ is relatively prime to $e_n +1$.
	The Coxeter exponents for the crystallographic Coxeter groups are
	$$\begin{array} {c|l}
	\text{type of $W$}& e_1,e_2,\ldots,e_n\\
	\hline
	A_n & 1,2,3,\ldots, n\\
	B_n & 1,3,5,\ldots, 2n-1\\
	D_n & n-1,1,3,5,\ldots, 2n-3\\
	E_6 & 1,4,5,7,8,11\\
	E_7 & 1,5,7,9,11,13,17\\
	E_8 & 1,7,11,13,17,19,23,29\\
	F_4 &1,5,7,11\\
	G_2 & 1,5\\
	\end{array}$$
	The classical Catalan number corresponds to type $A_n$.
	We can now define a $q$-Fibonacci analog as follows
\[ C_{W,{\mathcal F}}(a) = \prod_{i=1}^n \frac{[F_{a+e_i}]}{[F_{e_i+1}]}.\]
	We have computationally checked that this is a polynomial with positive integer coefficients when $a$ and $e_n +1$ are relatively prime, for each type and various values of $a$. It is interesting to note that although in type $A_n$
	we have shown that it is a polynomial as long as $(F_a,F_{e_n+1})=1$, for other types we must have the stronger condition $(a,e_n+1)=1$. For example $C_{F_4,{\mathcal F}}(2)$ is not a polynomial.
\end{Remark}

\section{The model of Bennett, Carrillo, Machacek, and Sagan}
Recently, Bennett, Carrillo, Machacek, and Sagan~\cite{combinatorial_bennett_2018} gave another combinatorial interpretation of the Lucas analog of the Binomial coefficient. A special case of their interpretation is that the Fibonomial number $\binom{n}{k}_\F=\frac{F_n^!}{F_k^! \cdot F_{n-k}^!}$ counts the number of certain partial tilings of a Young diagram of staircase shape of size $n$.
Similarly as above, we can also assign $q$-weights and elliptic weights to this model, giving rise to other combinatorial interpretations for the $q$-analog and elliptic analog of the Fibonomial number, respectively.

We start by briefly recalling the Bennett--Carrillo--Machacek--Sagan model. For $n \in \N$, consider the Young diagram of staircase shape of size $n$ (in French notation) associated to the partition $(n-1,n-2,\dots, 2,1)$: the position of the south-west corner is placed at the origin of the Cartesian coordinate system and the number of boxes in row $i$ (counted from bottom to top) is $n-i$. We also include the lines from $(0,n-1)$ to $(0,n)$ and from $(n-1,0)$ to $(n,0)$ for convenience. An example of a Young diagram of staircase shape for $n=9$ is illustrated on the left of Fig.~\ref{fig_bennetttiling}.

Given two natural numbers $n\geq k$, an \emph{$(n,k)$-tiling} is a partial tiling of a staircase shape of size~$n$ with monominos and dominos where one lattice path is specified such that:
\begin{itemize}\itemsep=0pt
	\item the lattice path goes from $(k,0)$ to $(0,n)$ using steps $(0,1)$ and $(-1,0)$. It stays inside the Young diagram and every west step $(-1,0)$ must be followed by a north step $(0,1)$.
	\item if a north step is not forced by a previous west step, the boxes that are on its left in the same row are tiled with monominos and horizontal dominos.
	\item if a north step is forced by a previous west step, the boxes that are on its right in the same row are tiled with monominos and horizontal dominos. The tile touching the north step is forced to be a domino (we call it a special domino).
\end{itemize}

We denote by $\staircaseTilings{n}{k}$ the collection of all $(n,k)$-tilings. An example for $n=9$ and $k=5$ is shown on the right of Fig.~\ref{fig_bennetttiling}.

\begin{figure}[h]	\centering
\begin{tikzpicture}[scale=0.7]
\grid{8}{1}
\grid{7}{2}
\grid{6}{3}
\grid{5}{4}
\grid{4}{5}
\grid{3}{6}
\grid{2}{7}
\grid{1}{8}
\draw (9,0) -- (0,0) -- (0,9);
\end{tikzpicture} \quad
\begin{tikzpicture}[scale=0.7]
%\node at (0.5,-0.5) {\color{white} a}; % this is just a trick to move the first figure so that both figures are at the same hegiht (I did not know how to do it differently)
%grid
\grid{8}{1}
\grid{7}{2}
\grid{6}{3}
\grid{5}{4}
\grid{4}{5}
\grid{3}{6}
\grid{2}{7}
\grid{1}{8}
\draw (9,0) -- (0,0) -- (0,9);
%monominos
\monomino{1}{2}
\monomino{4}{2}
\monomino{6}{3}
\monomino{3}{4}
\monomino{1}{7}
\monomino{2}{7}
%\monomino{1}{1}
%
%horizontal dominos:
\horizontalDomino{3}{2}
\horizontalDomino{2}{4}
\horizontalDomino{2}{6}
\horizontalDomino{8}{1}

%\verticalDomino{5}{2}

%special horizontal dominos:
\specialhorizontalDomino{6}{1}
\specialhorizontalDomino{5}{3}
\specialhorizontalDomino{4}{5}

%lattice path:
\draw[line width=\pathwidth,color=blue]
(5,0)
-- ++(-1,0)
-- ++(0,2)
-- ++(-1,0)
-- ++(0,2)
-- ++(-1,0)
-- ++(0,3)
-- ++(-1,0)
-- ++(0,1)
-- ++(-1,0)
-- ++(0,1);
\end{tikzpicture}

\caption{A Young diagram of staircase shape of size $9$ (left) and a $(5,9)$-tiling (right). }\label{fig_bennetttiling}
\end{figure}
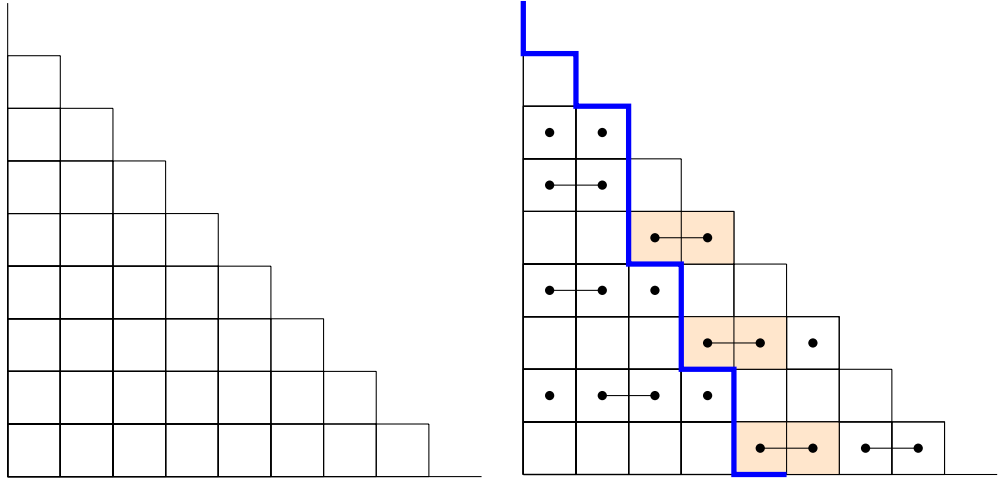

\begin{figure}[t]\centering
\tikzset{>=latex}
\begin{tikzpicture}[scale=1.1]
%\node at (0.5,-0.5) {\color{white} a}; % this is just a trick to move the first figure so that both figures are at the same hegiht (I did not know how to do it differently)
%grid
\grid{8}{1}
\grid{7}{2}

\draw (9,0) -- (0,0);
%monominos
%\monomino{1}{2}
%\monomino{4}{2}

%
%horizontal dominos:
\coloredhorizontalDomino{3}{2}{black}
\coloredhorizontalDomino{8}{1}{gray}

%special horizontal dominos:
%\specialhorizontalDomino{6}{1}

%lattice path:
\draw[line width=\pathwidth,color=blue]
(5,0)
-- ++(-1,0)
-- ++(0,2);

%weight distances
\draw[line width=1, color=gray, <->]
(8,0.65) -- node[above=-0.05cm] {\scriptsize floor} (6,0.65);
\draw[line width=1, color=gray, <->]
(0,0.65) -- node[above=-0.1cm] {\scriptsize height-1} (4,0.65);
\draw[line width=1, color=black, <->]
(7,1.65) -- node[above=-0.1cm] {\scriptsize height-1} (4,1.65);
\draw[line width=1, color=black, <->]
(0,1.65) -- node[above=-0.05cm] {\scriptsize floor} (3,1.65);

\end{tikzpicture}
\caption{An illustration of the height and floor statistics of one domino on the left and one domino on the right of the path. The domino on the top row has floor $3$ and height $4$, while the domino on the bottom row has floor $2$ and height~$5$.}	\label{fig_bennettweight}
\end{figure}
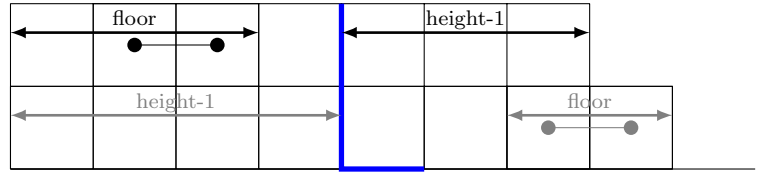

The tilings in the Bennett--Carrillo--Machacek--Sagan model are in natural bijection with the tilings of the Sagan--Savage model: the tilings on the right (resp. left) of the path in the staircase become the tilings of the columns below (resp. rows above) the path in the rectangle (see last paragraph before Section~3 in~\cite{combinatorial_bennett_2018}). We have chosen our example to match that of Fig.~\ref{fig_pathdominotiling}.

Using this bijection, one can easily reinterpret the $q$-weights and elliptic weights described in Sections~\ref{sec_q} and~\ref{sec_elliptic} for this new model.
Let $T$ be an $(n,k)$-tiling. For a tile $t$ (monomino or domino) on the left of the path, we define the \emph{floor} of $t$ as the number of boxes in the same row from the western border of the Young diagram to the eastern border of $t$, and the \emph{height} as 1 plus the number of boxes in the same row from the eastern border of the Young diagram to the north step to the east of~$t$.
For a tile $t$ on the right of the path the statistics are mirrored as illustrated in Fig.~\ref{fig_bennettweight}.

The $q$-weights of the tiles are defined as follows	
\begin{center}
\begin{tabular}{ccc}
 $\nu \left( \,
\begin{tikzpicture}[x=\boxsize,y=\boxsize,baseline=-4.5mm]
\monomino{0}{0}
\end{tikzpicture}
\, \right) = 1$,
& \
 $\nu \left( \,
\begin{tikzpicture}[x=\boxsize,y=\boxsize,baseline=-4.5mm]
\horizontalDomino{0}{0}
\end{tikzpicture}
\, \right) = q^{F_{\text{floor}}F_{\text{height}}}$,
& \
 $\nu \left( \,
\begin{tikzpicture}[x=\boxsize,y=\boxsize,baseline=-4.5mm]
\specialhorizontalDomino{0}{0}
\end{tikzpicture}
\, \right) = q^{F_{\text{floor}}F_{\text{height}+1}}$,
\end{tabular}
\end{center}
where the shaded domino represents a special domino forced by a west step of the path.
In the elliptic case the weights are defined as
\begin{center}
\begin{tabular}{ccc}
 $\widetilde\nu \left( \,
\begin{tikzpicture}[x=\boxsize,y=\boxsize,baseline=-4.5mm]
\monomino{0}{0}
\end{tikzpicture}
\, \right) = 1$,
& \
 $\widetilde\nu \left( \,
\begin{tikzpicture}[x=\boxsize,y=\boxsize,baseline=-4.5mm]
\horizontalDomino{0}{0}
\end{tikzpicture}
\, \right) = \omega_1(\text{floor},\text{height})$,
& \
 $\widetilde\nu \left( \,
\begin{tikzpicture}[x=\boxsize,y=\boxsize,baseline=-4.5mm]
\specialhorizontalDomino{0}{0}
\end{tikzpicture}
\, \right) = \omega_2(\text{floor},\text{height})$.
\end{tabular}
\end{center}
The $q$- and elliptic weights of a tiling~$T$ are defined as the product of the weights of its tiles.

As consequences of Theorem~\ref{thm_qFibonomial} and Theorem~\ref{thm_ellipticFibonomial} we get the following two results.

\begin{Corollary}\label{cor_q_Bennet}
	For $n\geq k$, the $q$-analog of the Fibonomial number can be computed as
	\[
	\BigqFibonomialBennet{n}{k} = \sum_{T\in \staircaseTilings{n}{k}} \nu(T).
	\]
\end{Corollary}

\begin{Corollary}\label{cor_elliptic_Bennet}
	For $n\geq k$, the elliptic analog of the Fibonomial number can be computed as
	\[
	\BigellipticFibonomialBennet{n}{k} = \sum_{T\in \staircaseTilings{n}{k}} \widetilde\nu(T).
	\]
\end{Corollary}

An example of Corollary~\ref{cor_q_Bennet}, for $n=4$ and $k=2$, is given in Fig.~\ref{fig_42tilings_Bennet}. We have chosen the order of the tilings to match the order of the tilings for the corresponding example in Fig.~\ref{fig_pathdominotilings22}.
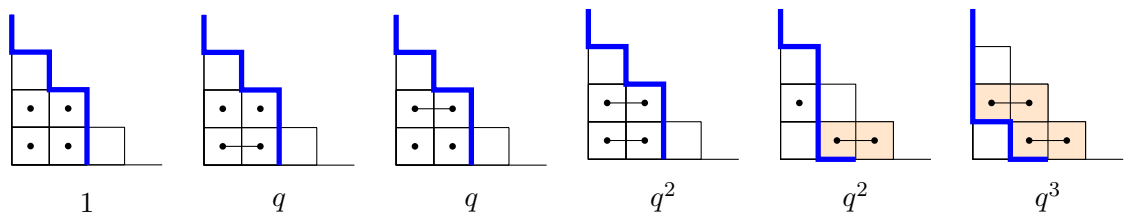
\begin{figure}[h]	\centering
%\begin{subfigure}[b]{.15\linewidth}
	\begin{tikzpicture}[scale=0.5]
	\grid{3}{1}
	\grid{2}{2}
	\grid{1}{3}
	\draw (4,0) -- (0,0) -- (0,4);
	\monomino{1}{1}
	\monomino{1}{2}
	\monomino{2}{1}
	\monomino{2}{2}
	\draw[line width=\pathwidth,color=blue] (2,0) -- ++(0,2) -- ++(-1,0) -- ++(0,1) -- ++ (-1,0) -- ++(0,1);
	\node at (2,-1) {$1$};
	\end{tikzpicture}
%\end{subfigure}
\quad
%\begin{subfigure}[b]{.15\linewidth}
	\begin{tikzpicture}[scale=0.5]
	\grid{3}{1}
	\grid{2}{2}
	\grid{1}{3}
	\draw (4,0) -- (0,0) -- (0,4);
	\horizontalDomino{2}{1}
	\monomino{1}{2}
	\monomino{2}{2}
	\draw[line width=\pathwidth,color=blue] (2,0) -- ++(0,2) -- ++(-1,0) -- ++(0,1) -- ++ (-1,0) -- ++(0,1);
	\node at (2,-1) {$q$};
	\end{tikzpicture}
%\end{subfigure}
\quad
%\begin{subfigure}[b]{.15\linewidth}
	\begin{tikzpicture}[scale=0.5]
	\grid{3}{1}
	\grid{2}{2}
	\grid{1}{3}
	\draw (4,0) -- (0,0) -- (0,4);
	\horizontalDomino{2}{2}
	\monomino{1}{1}
	\monomino{2}{1}
	\draw[line width=\pathwidth,color=blue] (2,0) -- ++(0,2) -- ++(-1,0) -- ++(0,1) -- ++ (-1,0) -- ++(0,1);
	\node at (2,-1) {$q$};
	\end{tikzpicture}
%\end{subfigure}
\quad
%\begin{subfigure}[b]{.15\linewidth}
	\begin{tikzpicture}[scale=0.5]
	\grid{3}{1}
	\grid{2}{2}
	\grid{1}{3}
	\draw (4,0) -- (0,0) -- (0,4);
	\horizontalDomino{2}{1}
	\horizontalDomino{2}{2}
	\draw[line width=\pathwidth,color=blue] (2,0) -- ++(0,2) -- ++(-1,0) -- ++(0,1) -- ++ (-1,0) -- ++(0,1);
	\node at (2,-1) {$q^2$};
	\end{tikzpicture}
%\end{subfigure}
\quad
%\begin{subfigure}[b]{.15\linewidth}
	\begin{tikzpicture}[scale=0.5]
	\grid{3}{1}
	\grid{2}{2}
	\grid{1}{3}
	\draw (4,0) -- (0,0) -- (0,4);
	\specialhorizontalDomino{3}{1}
	\monomino{1}{2}
	\draw[line width=\pathwidth,color=blue] (2,0) -- ++(-1,0) -- ++(0,2) -- ++(0,1) -- ++ (-1,0) -- ++(0,1);
	\node at (2,-1) {$q^2$};
	\end{tikzpicture}
%\end{subfigure}
\quad
%\begin{subfigure}[b]{.15\linewidth}
	\begin{tikzpicture}[scale=0.5]
	\grid{3}{1}
	\grid{2}{2}
	\grid{1}{3}
	\draw (4,0) -- (0,0) -- (0,4);
	\specialhorizontalDomino{3}{1}
	\specialhorizontalDomino{2}{2}
	\draw[line width=\pathwidth,color=blue] (2,0) -- ++(-1,0) -- ++(0,1) -- ++ (-1,0) -- ++(0,3);
	\node at (2,-1) {$q^3$};
	\end{tikzpicture}
%\end{subfigure}

	\caption{The six $(4,2)$-tilings of a staircase shape of size 4, and their contribution to the $q$-Fibonomial
		$\protect\qFibonomialBennet{4}{2}=\frac{[F_4][F_3]}{[F_2][F_1]}=\frac{[3][2]}{[1][1]}=1+2q+2q^2+q^3$
		when computed as a generating function $\sum\limits_{T\in \staircaseTilings{4}{2}} \nu(T)$.}
	\label{fig_42tilings_Bennet}
\end{figure}

\begin{Remark}
	The Bennett--Carrillo--Machacek--Sagan model also gives rise to combinatorial interpretations of the Lucas analog of the (ordinary) Catalan numbers and Fuss--Catalan numbers described in \cite{combinatorial_bennett_2018}.
	They define a Catalan partial tiling $C$ of a Young diagram of staircase shape of size $2n$ to be a $(2n,n-1)$-tiling where the first row (at the bottom) stays blank except for the forced special domino if the first step is a west step.
	
The Lucas analog of the Catalan numbers
	\[ C_{\{n\}} := \frac{1}{\{n+1\}} \frac{\{2n\}!}{\{n\}! \{n\}!}, \] where $\{n\}!=\prod\limits_{k=1}^{n} \{k\}$,
	is the sum over the weights of all Catalan partial tilings $C$ of size $2n$ if all dominos have weight $t$, all monominos have weight $s$, and the weight of the tiling $C$ is defined as the product of the weights of its tiles \cite[Corollary~4.2.]{combinatorial_bennett_2018}.
	
	This model does not translate to the $q$-Fibonacci analog of the Catalan numbers \[ [\nFCat{n}] := \frac{1}{[F_{n+1}]_q} \BigqFibonomial{2n}{n} = \frac{\qFibofactorial{2n}}{\qFibofactorial{n+1}\qFibofactorial{n}}. \] For example, $[\nFCat{3}]$ is not equal to the sum of the weights of all Catalan partial tilings $C$ of size~$6$ with the corresponding $q$-weights defined above.
\end{Remark}

Although the Bennett--Carrillo--Machacek--Sagan model does not seem to trivially solve the following problem, it still might be easier than Open Problem~\ref{prob_qFib_rational}.
\begin{Openproblem}	
	Find a combinatorial interpretation of the $q$-Fibonacci analog of the Catalan numbers (and the Fuss--Catalan numbers).
\end{Openproblem}

\subsection*{Acknowledgements}
The authors are grateful to Farid Aliniaeifard, Tom Denton, Shu Xiao Li, Drew Armstrong, Bruce Sagan, Michael Schlosser, and Mike Zabrocki for helpful and inspiring discussions. We also thank anonymous referees for their useful comments and suggestions.
NB was supported by NSERC and a York Research Chair. CC was supported by the Austrian Science Foundation FWF, grant F 5008-N15, in the framework of the Special Research Program Algorithmic and Enumerative Combinatorics. JK was supported by the Austrian Science Foundation FWF, grant P 32305.

\pdfbookmark[1]{References}{ref}
\LastPageEnding

\end{document}